\documentclass[12pt]{amsart}
\usepackage[top=30truemm,bottom=30truemm,left=25truemm,right=25truemm]{geometry}
\usepackage{mathrsfs}
\usepackage{tikz-cd}

\usepackage{color}
\usepackage{bm}
\usepackage{amsfonts,amssymb}
\usepackage{dsfont}
\usepackage{amscd}
\usepackage{extarrows}
\usepackage{amsmath}
\usepackage{mathrsfs}
\usepackage{enumerate}
\usepackage{amscd}
\usepackage[all]{xy}
\usepackage[pagebackref,colorlinks]{hyperref}

\theoremstyle{plain} 
\newtheorem{theorem}{\indent\bf Theorem}[section]

\theoremstyle{definition} 

\newtheorem{prob}[theorem]{\indent\bf Problem}

\newtheorem{thm}{Theorem}[section]
\newtheorem{cor}[thm]{Corollary}
\newtheorem{lem}[thm]{Lemma}

\theoremstyle{definition}
\newtheorem{defn}{Definition}[section]

\theoremstyle{remark}
\newtheorem{rem}{Remark}[section]

\newcommand{\be}{\begin{equation}}
	\newcommand{\ee}{\end{equation}}
\newcommand{\bea}{\begin{eqnarray}}
	\newcommand{\eea}{\end{eqnarray}}
\newcommand{\ben}{\begin{eqnarray*}}
	\newcommand{\een}{\end{eqnarray*}}
\newcommand{\bt}{\begin{split}}
	\newcommand{\et}{\end{split}}
\newcommand{\bet}{\begin{equation}}
	\newcommand{\mc}{\mathbb{C}}
	
	\newcommand{\ra}{\rightarrow}

\newcommand{\mo}{\mathcal{O}}
\newcommand{\bc}{\begin{center}}
\newcommand{\ec}{\end{center}}
\newcommand{\bi}{\begin{itemize}}
\newcommand{\ei}{\end{itemize}}

\newcommand{\h}{\textbf}

\begin{document}
		\title[Uniqueness of irreducible desingularization]
{Uniqueness of irreducible desingularization of singularities associated to negative vector bundles}
		
		\author[F. Deng]{Fusheng Deng}
		\address{Fusheng Deng: \ School of Mathematical Sciences, University of Chinese Academy of Sciences\\ Beijing 100049, P. R. China}
		\email{fshdeng@ucas.ac.cn}
		
	\author[Y. Li]{Yinji Li}
	\address{Yinji Li:  Institute of Mathematics\\Academy of Mathematics and Systems Sciences\\Chinese Academy of
		Sciences\\Beijing\\100190\\P. R. China}
	\email{1141287853@qq.com}
	\author[Q. Liu]{Qunhuan Liu}
		\address{Qunhuan Liu: \ School of Mathematical Sciences, University of Chinese Academy of Sciences\\ Beijing 100049, P. R. China}
		\email{liuqunhuan23@mails.ucas.edu.cn}

		\author[X. Zhou]{Xiangyu Zhou}
		\address{Xiangyu Zhou: Institute of Mathematics\\Academy of Mathematics and Systems Sciences\\and Hua Loo-Keng Key
			Laboratory of Mathematics\\Chinese Academy of
			Sciences\\Beijing\\100190\\P. R. China}
		\address{School of
			Mathematical Sciences, University of Chinese Academy of Sciences,
			Beijing 100049, P. R. China}
		\email{xyzhou@math.ac.cn}
		
\begin{abstract}
We prove that the irreducible desingularization of a singularity given by the Grauert blow down of a negative holomorphic vector bundle
over a compact complex manifold is unique up to isomorphism, and as an application, we show that two negative line bundles
over compact complex manifolds are isomorphic if and only if their Grauert blow downs have isomorphic germs near the singularities.
We also show that there is a unique way to modify a submanifold of a complex manifold to a hypersurface, namely,
the blow up of the ambient manifold along the submanifold.
\end{abstract}
		
		\thanks{}

		\maketitle

\section{Introduction}
The present work grows up from an effort for studying singularities given by blowing down zero sections of negative holomorphic vector bundles over compact complex manifolds.
We start our discussion by considering general modifications.

Modification of complex spaces, a concept originally proposed by Behnke and Stein in \cite{B-S51}, is a very important type of surgery in the study of bimeromorphic geometry of complex spaces.
The definition of modification is as follows.

\begin{defn}\label{def:modification}
A proper surjective holomorphic map $\pi:X\ra Y$ of complex spaces is called a modification if there are analytic subsets $A\subset X$ and $B\subset Y$
such that:
\begin{itemize}
\item[(1)] $B=\pi(A)$,
\item[(2)] $\pi|_{X\backslash A}:X\backslash A\ra Y\backslash B$ is biholomorphic,
\item[(3)] $A$ and $B$ are analytically rare, and
\item[(4)] $A$ and $B$ are minimal with the properties (1)-(3).
\end{itemize}
\end{defn}
We recall that an analytic set $A$ in a complex space $X$ is called analytically rare if for every open set $U\subset X$ the restriction map $\mo_X(U)\ra\mo_X(U\backslash A)$
is injective. If $X$ is reduced, this is equivalent to say that $A$ does not contain any irreducible component of $X$.
There are some general results about modifications.
For example, Grauert showed that a connected analytic subset $A$ in an irreducible complex space
can be "modified" to a point if and only if $A$ admits a strictly pseudoconvex neighborhood in $X$ \cite{Gr62},
and Hironaka showed that a modification can always be dominated by a locally finite sequence of blow-ups along analytic subsets \cite{Hor75}.
Please take a look at \cite{Pet94} for a very nice survey on modifications.

For a modification $\pi:(X,A)\ra (Y,B)$, we always refer to a modification of the form as in Definition \ref{def:modification}.

\begin{defn}\label{def:ismomorphic modification}
Two modifications $\pi:(X,A)\ra (Y,B)$ and $\pi':(X',A')\ra (Y,B)$ are called isomorphic if there is a biholomorphic map
$\sigma:(X, A)\ra (X', A')$ such that $\pi'\circ\sigma=\pi$.
\end{defn}

\begin{defn}\label{def:modification}
A modification $\pi:(X,A)\ra (Y,B)$ is called a \h{desingularization} of $(Y, B)$ if $X$ is regular;
and called an \h{irreducible desingularization} of $(Y, B)$ if $X$ is regular and $A$ is irreducible.
\end{defn}

In the present article, we are interested in the general problem of fundamental importance that,
for a given space $(Y, B)$, classifying all irreducible desingularizations $\pi:(X, A)\ra (Y,B)$ of $(Y,B)$ up to isomorphisms.
It seems that little is known about the above classification problem in the literatures.
As far as we know,  the classification problem is still open
even for the special case that $B$ is just an isolated regular point of $Y$.
In this situation,  of course one does not need to desingularize $Y$ along $B$,
so the classification problem for such $(Y,B)$ can be viewed as the ``kernel" of the above general classification problem.

In this work, we make some progress in this direction.
We will consider those cases that $B$ is a single point or a smooth submanifold.
The first result is about the uniqueness of the way for modifying a regular point in a complex space to an irreducible hypersurface.
Through out this paper, a hypersurface of a complex manifold always means an analytic subset of codimension 1.

\begin{thm}\label{thm: modiy a regular point}
Let $f:(X, E)\ra (\mc^n, 0)$ be an irreducible desingularization of $(\mc^n, 0)$.
We have the following two cases:
\bi
\item[(1)] If $E$ is a single point, then $f:(X, E)\ra (\mc^n, 0)$ is a biholomorhic map.
\item[(2)] If $\dim E>0$, then $f:(X, E)\ra (\mc^n, 0)$ is isomorphic to the modification
$$(\mathcal{O}_{\mathbb{P}^{n-1}}(-1),\mathbb{P}^{n-1})\ra (\mc^n,0),$$
which is the blow-up of $\mathbb{C}^n$ along the origin $\{0\}$.
\ei
\end{thm}

The assumption on the irreducibility of $E$ is obviously necessary,
since we can proceed blow-ups successively to get different modifications with reducible exceptional divisors.

The conclusion in Theorem \ref{thm: modiy a regular point} is of course local in nature.
So indeed the following result holds.
Let $X, Y$ be complex manifolds, with $E\subset X$ being an irreducible analytic subset of positive dimension and $y_0\in Y$.
If there exists a proper holomorphic map $f:X\ra Y$ such that $f(E)=\{y_0\}$ and $f:X\setminus E\ra Y\setminus\{y_0\}$ is biholomorphic,
then $f:(X,E)\ra (Y,y_0)$ is isomorphic to the blow-up of $Y$ along $y_0$.

Theorem \ref{thm: modiy a regular point} has several generalizations.
The first one is about uniqueness of the irreducible desingularization of an isolated cone singularity.
By definition, a cone $Y$ in $\mc^{n+1}$ is an analytic subset defined by a series of homogenous polynomials on $\mc^{n+1}$.
We denote by $Y_P$ the image $p(Y\backslash\{0\})$ of $Y\backslash\{0\}$ under the natural projection $p:\mc^{n+1}\backslash\{0\}\ra \mathbb P^{n}$ from $\mc^{n+1}\backslash\{0\}$
to the complex projective space $\mathbb P^n$ of dimension $n$.
Then it is obvious that $Y\backslash\{0\}$ is regular if and only if $Y_P$ is a submanifold of $\mathbb P^n$.

Recall that the blow up of $\mc^{n+1}$ at $0$ can be identified with the total space of the tautological line bundle $\mo_{\mathbb P^n}(-1)$
with the nature projective $\pi:\mo_{\mathbb P^n}(-1)\ra\mc^{n+1}$.
For an irreducible analytic subset $Z$ of $\mc^{n+1}$, the \emph{strict transform} of $Z$ is defined to be the minimal analytic subset $\tilde Z$
in $\mo_{\mathbb P^n}(-1)$ that contains $\pi^{-1}(Z\backslash\{0\})$.

Theorem \ref{thm: modiy a regular point} can be generalized as follows.

\begin{thm}\label{thm:cone sigularity}
With notations as above, let $f:(X,E)\ra (Y,0)$ be an irreducible desingularization of $(Y,0)$,
with $E\subset X$ being an irreducible hypersurface,
and $Y\subset\mc^{n+1}$ being a cone with $0$ as an isolated singularity,
then $f:(X,E)\ra (Y,0)$ is isomorphic to the modification $\pi:(\tilde Y, \pi^{-1}(0)\cap\tilde Y=Y_P)\ra (Y,0)$.
\end{thm}

We now discuss the relations of Theorem \ref{thm: modiy a regular point} and Theorem \ref{thm:cone sigularity} with the theory of ample vector bundles.
Let $\pi: F\ra M$ be a holomorphic verctor bundle over a compact complex manifold $M$.
Recall that $F$ is said to be ample (resp. very ample) if the line bundle $\mo_{\mathbb P(F^*)}(1)$ on the projectivization $\mathbb P(F^*)$ is ample (very ample),
where $F^*$ is the vector bundle dual to $F$.
We say that $F$ is \h{negative (in the sense of Grauert)} if $F^*$ is ample.
We can identify $M$ with the zero section of $F$ and then view $M$ as a submanifold of $F$.
By a remarkable result of Grauert in \cite{Gr62}, \h{$F$ is negative if and only if $M$ is exceptional in $F$},
that is, there is a modification $f:(F, M)\ra (Z, z_0)$, where $Z$ is a complex space and $z_0\in Z$ is a point. Such a modification is unique up to normalization.
In other words, $F$ is negative if and only if $M\subset F$ can be modified to a single point.
For convenience, we will call $(Z, z_0)$ the \h{Grauert blow down} of $F$.
In general $z_0$ is a singularity of $Z$.
Indeed, a consequence of Theorem \ref{thm: modiy a regular point} is the following result.

\begin{thm}\label{thm:l.b. blow down regular}
Let $\pi: F\ra M$ be a negative holomorphic vector bundle over a compact complex manifold $M$.
 Then the Grauert blow down $(Z, z_0)$ of $F$ is regular at the point $z_0$ if and only if $F$ is a line bundle,
 $M\cong \mathbb P^n$ for some $n$,
 and as holomorphic line bundles $\pi: F\ra M$ is isomorphic to the tautological line bundle $\mo_{\mathbb P^n}(-1)\ra \mathbb P^n$.
\end{thm}

The case of Theorem \ref{thm:l.b. blow down regular} that $M$ is a compact Riemann surface and $F^*$ is a very ample line bundle is due to Morrow and Rossi \cite{MR79}.

Let $L\ra M$ be a line bundle with $L^*$ being very ample,
and let $\phi:M\ra\mathbb P^N$ be the embedding of $M$ into the projective space $\mathbb P^N$
given by $L^*$, where $N=\text{dim}H^0(M,L^*)-1$.
Then the Grauert blow down $(Z, z_0)$ of $L$ is isomorphic to a cone in $\mc^{N+1}$ given by
$\pi^{-1}(\phi(M))\cup \{0\}$, up to normalization, where $\pi:\mc^{N+1}\setminus\{0\}\ra\mathbb P^N$ is the canonical projection.
With this observation, we can reformulate Theorem \ref{thm:cone sigularity} as follows:
any irreducible desingularization $f:(X,E)\ra (Z, z_0)$ with $E$ being a hypersurface
must be isomorphic to the Grauert blow down $(L, M)\ra (Z, z_0)$.

Now it is very natural to ask if the above result can be generalized to general negative line bundles.
It is surprising that the answer is affirmative.
The heart of the present work is the following theorem.

\begin{thm}\label{thm:ample line bundles}
Let $\pi:L\ra M$ be a negative line bundle over a compact complex manifold $M$, and let $(Z, z_0)$ be the Grauert blow down of $L$.
Then any irreducible desingularization $f:(X,E)\ra (Z, z_0)$ with $E$ being a hypersurface
must be isomorphic to the Grauert blow down $(L, M)\ra (Z, z_0)$ itself.
\end{thm}

Theorem \ref{thm:ample line bundles} provides a new perspective about negative/ample line bundles:
\h{the total spaces of negative line bundles are the unique irreducible desingularizations of a special type of isolated singularities}.
It also naturally leads to the follow problems.

\begin{prob} \
\bi
\item[(1)] characterize isolated singularities that can be realized as Grauert blow down of negative line bundles;
\item[(2)] characterize isolated singularities that admit irreducible desingularizations, and classify their irreducible desingularizations;
\item[(3)] characterize isolated singularities that admit one and only one irreducible desingularizations (with exceptional set being hypersurfaces).
\ei
\end{prob}

Theorem \ref{thm:ample line bundles} can be generalized to holomorphic vector bundles.
Let $\pi:F\ra M$ be a holomorphic vector bundle of rank $r$ over a compact complex manifold $M$.
If $F$ is negative, the tautological line bundle $\mo_{\mathbb P(F^*)}(-1)$ over the projectivization
$\mathbb P(F^*)$ of $F^*$ is negative.
The total space of $\mo_{\mathbb P(F^*)}(-1)$ can be identified with the blow up of $F$ along the zero section $M$.
So the Grauert blow downs of $F$ and $\mo_{\mathbb P(F^*)}(-1)$ are isomorphic as complex spaces.
In connection with this and Theorem \ref{thm:ample line bundles}, we have the following

\begin{thm}\label{thm:ample vector bundles}
Let $f:(X,S)\ra (Z, z_0)$ be an irreducible desingularization of the Grauert blow down $(Z, z_0)$ of a negative line bundle $\pi:L\ra M$
over a compact complex manifold $M$, with $S$ being a smooth submanifold of $X$.
Then there exists a negative vector bundle $F\ra S$ such that $(Z, z_0)$ is isomorphic to the Grauert blow down of $F$
and  $f:(X,S)\ra (Z, z_0)$ is isomorphic to the Grauert blow down $(F, S)\ra (Z, z_0)$.
\end{thm}

We do not know if the conclusion in Theorem \ref{thm:ample vector bundles} still holds
if $S$ is not assumed to be regular.

Theorem \ref{thm:ample line bundles} has the following interesting consequence.

\begin{thm}\label{thm:blow down determine l.b.}
Let $\pi:L\ra M$ and $\pi':L'\ra M'$ be two negative line bundles over compact complex manifolds.
Let $(Z, z_0)$ and $(Z', z'_0)$ be the Grauert blow downs of $L$ and $L'$ respectively.
If there exist neighborhoods $U$ of $z_0$ in $Z$ and $U'$ of $z'_0$ in $Z'$ such that $U\cong U'$ as complex spaces,
then $M\cong M'$, and $L\cong L'$ as holomorphic line bundles.
\end{thm}

Let $X$ be a complex manifold and $S\subset X$ be a compact complex submanifold.
It is known that $S$ is exceptional in $X$ if its normal bundle $N_{X/S}$ of $S$ in $X$ is negative \cite{Gr62}.
With this assumption, one may wonder if the Grauert blow downs of $(X,S)$ and $N_{X/S}$ are isomorphic as complex spaces.
From Theorem \ref{thm:ample line bundles}, we can get the answer to this question as follows.

\begin{cor}\label{cor:blow down nbd and normal bundle equi}
Let $X$ be a complex manifold and $S\subset X$ be a compact smooth hypersurface
such that the normal bundle $N_{X/S}$ is negative.
Let $(Z,z_0)$ and $(Z',z_0')$ be the normal Grauert blow downs of $(X,S)$ and $(N_{X/S},S)$.
Then $z_0$ and $z_0'$ have isomorphic neighborhoods in $Z$ and $Z'$ if and only if
$S$ has isomorphic neighborhoods in $X$ and $N_{X/S}$.
\end{cor}

Finally we generalize Theorem \ref{thm: modiy a regular point} to the uniqueness of
modifying a submanifold.

\begin{thm}\label{thm:modify subumanifolds}
Let $X$ be a complex manifold and $S\subset X$ be a complex submanifold.
Then any irreducible desingularization $f:(Y, E)\ra (X,S)$ of $(X,S)$ must be isomorphic to the
blow up of $X$ along $S$, provided that $E$ is a hypersurface.
\end{thm}

We do not know if the conclusion in Theorem \ref{thm:modify subumanifolds} still holds
if $S$ is assumed to be an irreducible analytic subset of $X$.

The results in the present work and their proofs are analytic in nature.
It seems interesting to establish the algebraic geometric counterparts of them.

The remaining of the paper is arranged as follows.
We give the proof of Theorem \ref{thm: modiy a regular point} and Theorem \ref{thm:l.b. blow down regular} in \S \ref{sec:modify point}.
Though Theorem \ref{thm: modiy a regular point} can be covered by Theorem \ref{thm:ample line bundles} or Theorem \ref{thm:modify subumanifolds},
we also include its proof for showing more clearly the basic idea in our arguments.
Theorem \ref{thm:cone sigularity} is covered by Theorem \ref{thm:ample line bundles} and can be proved in the same way as the argument of Theorem \ref{thm: modiy a regular point}.
In \S \ref{sec:uniqueness for negative v.b.} we give the proof of Theorem \ref{thm:ample line bundles}, Theorem \ref{thm:ample vector bundles},
and Corollary \ref{thm:blow down determine l.b.}.  The proof of Theorem \ref{thm:ample line bundles} requires new insight beyond the idea in the proof of  Theorem \ref{thm: modiy a regular point}.
In the final \S \ref{sec:modify submanifolds}, we give the proof of Theorem \ref{thm:modify subumanifolds}.

\subsection*{Acknowledgements}
This research is supported by National Key R\&D Program of China (No. 2021YFA1003100),
NSFC grants (No. 11871451, 12071310), and the Fundamental Research Funds for the Central Universities.

\section{Uniqueness of modifying a regular point}\label{sec:modify point}
This section is to prove Theorem \ref{thm: modiy a regular point} and its consequence Theorem \ref{thm:l.b. blow down regular}.

\begin{thm}[=Theorem \ref{thm: modiy a regular point}]\label{thm: modify a regular point text}
Let $f:(X, E)\ra (\mc^n, 0)$ be an irreducible desingularization of $(\mc^n, 0)$.
We have the following two cases:
\bi
\item[(1)] If $E$ is a single point, then $f:(X, E)\ra (\mc^n, 0)$ is a biholomorhic map.
\item[(2)] If $\dim A>0$, then $f:(X, E)\ra (\mc^n, 0)$ is isomorphic to the modification
$$(\mathcal{O}_{\mathbb{P}^{n-1}}(-1),\mathbb{P}^{n-1})\ra (\mc^n,0),$$
which is the blow-up of $\mathbb{C}^n$ along $\{0\}$.
\ei
\end{thm}

\begin{rem}
$\mathcal{O}_{\mathbb{P}^{n-1}}(-1)$ can be viewed as the submanifold of $\mathbb{C}^n\times\mathbb{P}^{n-1}$:
\begin{align*}
\{\big((z_1,\cdots,z_n),[w_1:\cdots:w_n]\big)\in\mathbb{C}^n\times\mathbb{P}^{n-1}:z_iw_j=z_jw_i,1\leq i,j\leq n\}.
\end{align*}
The blow-up $\pi:\mathcal{O}_{\mathbb{P}^{n-1}}(-1)\rightarrow\mathbb{C}^n$ is induced by the projection $\pi:\mathbb{C}^n\times\mathbb{P}^{n-1}\rightarrow \mathbb{C}^n$.
\end{rem}
\begin{proof}
If $E$ is a single point, $f^{-1}:\mathbb{C}^n\backslash\{0\}\rightarrow X\backslash E$ holomorphically extends to $0$,
this implies that $f:(X,E)\rightarrow(\mathbb{C}^n,0)$ is a biholomorphism.

If $\dim E>0$, we first observe that $\dim E=n-1$. Since $f:X\rightarrow \mc^n$ degenerates on $E$, we get that
\begin{align*}
J_f^{-1}(0)=\{x\in X:\mathrm{the}\ \mathrm{Jacobian}\ \mathrm{of}\ f \ \mathrm{vanishes}\ \mathrm{at} \ x\}
\end{align*}
is non-empty, therefore $\dim J_{f}^{-1}(0)=n-1$ and it must coincide with $E$.

We set $f_i:=f^*z_i$, $i=1,\cdots,n$, the pull-back of the coordinate functions on $\mc^n$. For $\xi=(\xi_1,\cdots,\xi_n)\in \mc^n\setminus\{0\}$, we set
\begin{align*}
L=L_{\xi}&:=\{t\xi:t\in\mathbb{C}\},\\
C_0&:=f^{-1}(L),\\
C_1&:=\overline{C_0\backslash E}.
\end{align*}
It is clear that $C_1$ is an analytic subset in $X$.
In fact, $C_1$ is the irreducible component of $C_0$ that contains $C_0\backslash E$.
We claim that $C_1\cap E=\{p_L\}$ is a single point.
We give the proof of this claim.
The properness of $f$ yields that $C_1\cap E $ is non-empty.
On the other hand, it is clear that $C_1\cap E $ contains only finitely many points since $C_1$ is of dimension 1.
We can take some $p\in C_1\cap E$.
Since $f$ is injective on $X\backslash E$,
$f$ is non-constant on any irreducible component of the germ $(C_1,p)$ of $C_1$ at p.
It follows that $f$ maps any irreducible component the germ $(C_1,p)$ onto some neighborhood of $0$ in $L$,
hence $(C_1,p)$ is irreducible and $f:(C_1,p)\rightarrow (L,0)$ is surjective.
If there exists another $q\in C_1\cap E$, then $f:(C_1,q)\rightarrow (L,0)$ is also surjective,
contradicting with the fact that $f$ is injective on $X\backslash E$.

We have seen that $(C_1,p_L)$ is irreducible,  so $f:(C_1,p_L)\rightarrow (L,0)$ is a bijection,
whose inverse is also holomorphic by the Riemann's removable singularity theorem.
It then follows that the analytic curve $C_1$ is regular.
Moreover, $f|_{C_1}$ is non-degenerate at $p_L$, i.e. $df|_{C_1}(p_L)\neq0$.
This implies that $df_1,\cdots, df_n$ can not vanish simultaneously at $p_L$.
We assume without loss of generality that $df_1(p_L)\neq 0$, then $(f_1^{-1}(0),p_L)$ is a smooth complex hypersurface germ of $X$.
Since $E=f_1^{-1}(0)\cap\cdots\cap f_n^{-1}(0)$, we can infer that $(E,p_L)=(f_1^{-1}(0),p_L)$, and in particular $E$ is regular at $p_L$.
Otherwise, for some $j>1$, $f_j$ does not vanish on $(f^{-1}_1(0),p_L)$.
This leads to $\dim(E,p_L)\leq n-2$, a contradiction.
As a result, $\frac{f_2}{f_1},\cdots,\frac{f_n}{f_1}$ can be viewed as holomorphic functions defined near $p_L$.

Consider the holomorphic map
\begin{align*}
F:=\pi^{-1}\circ f:X\backslash E &\rightarrow \mathcal{O}_{\mathbb{P}^{n-1}}(-1) \\
x\ \ &\mapsto \big(f_1(x),\cdots,f_n(x),[f_1(x):\cdots:f_n(x)]\big),
\end{align*}
we want to show that $F$ can be extended to a biholomorphic map from $X$ to $\mathcal{O}_{\mathbb{P}^{n-1}}(-1)$.
Near $p_L$, we have
$$F(x)=\big(f_1(x),\cdots,f_n(x),[1: \frac{f_2}{f_1}(x):\cdots:\frac{f_n}{f_1}(x)]\big).$$
It follows that $F$ holomorphically extends to some neighbourhood of $p_L$ in $X$,
and $F(p_L)=(0,[L])\in \mc^n\times\mathbb P^{n-1}$, where $[L]\in\mathbb P^{n-1}$ denotes the point corresponding to the line $L$.
For each $q\in E$, if some of $df_1(q),\cdots,df_n(q)$ does not vanish, then $F$ holomorphically extends to $q$ in the same way.

Let
$$X'=\{x\in X|df_1(x), \cdots, df_n(x)\ \text{do not vanish simultaneously}\},$$
then $X'\subset X$ is a Zariski open subset that contains $X\backslash E$, and $X'\cap E\subset E_{\text{reg}}$.
Note that $X'$ contains $p_L$ for all linear subspace $L\subset\mc^n$ of dimension 1,
it follows that $\mathbb P^{n-1}$, identified with the subset set $\{0\}\times \mathbb P^{n-1}$ in $\mathcal{O}_{\mathbb{P}^n}(-1)$,
lies in $F(X')$.
This implies that $F:X'\rightarrow \mathcal{O}_{\mathbb{P}^n}(-1)$ is surjective.

We move to prove that $X'=X$ and $F$ is biholomorphic.
We first show that $F$ is non-degenerate on $X'$.
Note that the analytic subset
\begin{align*}
J:=\{x\in X':\mathrm{the}\ \mathrm{Jacobian}\ \mathrm{of}\ F \ \mathrm{vanishes}\ \mathrm{at} \ x\},
\end{align*}
of $X'$ is either empty or has pure dimension $(n-1)$ in $X'$.
By assumption, $E$ is irreducible, it follows that $X'\cap E$ is also irreducible since it is  a Zariski open subset set of $E$.
Note that $J\subset X'\cap E$, so $J$ must equals to $X'\cap E$ if it is non-empty.
We have seen that $F$ maps $X'\cap E$ onto $\mathbb{P}^{n-1}$.
Note that $X'\cap E$ is regular as explained above,
by Sard's theorem, the differential of $F|_{X'\cap E}$ is non-degenerate at some point on $X'\cap E$, saying $p_0$.
It follows that  $F|_{X'\cap E}$ is injective near $p_0$, which leads to that $F$ is injective on $X'$ near $p_0$ since $F$ is injective on $X\backslash E$ and maps $X\backslash E$ to $\mathcal{O}_{\mathbb{P}^n}(-1)\backslash \mathbb P^n$.
Then we can deduce from a classical result in several complex variables (see e.g. \cite[Theorem 8.5]{FG02}) that $p_0\notin J$.
As a consequence, $J$ is empty and $F$ is non-degenerate on $X'$.

Now, we can prove that $F:X'\rightarrow \mathcal{O}_{\mathbb{P}^{n-1}}(-1)$ is injective.
Suppose to the contrary, for some $p_1\neq p_2\in X'\cap E$, $F(p_1)= F(p_2)=q\in \mathbb{P}^{n-1}$.
 Then there exists some neighbourhood $U_i$ of $p_i$, some neighbourhood $V$ of $q$, such that $F:U_i\rightarrow V$ is a biholomorphism.
 This contradicts with the fact that $F$ is injective on $X\backslash E$.

Up till now, we have proved that $F:X'\rightarrow \mathcal{O}_{\mathbb{P}^{n-1}}(-1)$ is a non-degenerate, holomorphic bijection, hence a biholomorphism.
It immediately follows that $X'\cap E=F^{-1}(\mathbb{P}^{n-1})$ is compact,
which implies that $E=X'\cap E$ since $E$ is irreducible.
We get $X=X'$ and $F:X\ra \mathcal{O}_{\mathbb{P}^{n-1}}(-1)$ is a biholomorphic map.
\end{proof}

\begin{rem}
If $X$ is not assumed to be smooth, Theorem \ref{thm: modify a regular point text} fails in general.
Let $X$ be the blow-up of $\mathbb{C}^2$ along the ideal $\mathcal{I}=(z,w^2)$. That is
\begin{align*}
X=\{\big((z,w),[u:v]\big)\in\mathbb{C}^2\times\mathbb{P}^1:vz=uw^2\}.
\end{align*}
It is easy to see that $X$ has an isolated singularity $\big((0,0),[1:0]\big)$,
the exceptional set $E$ is $\mathbb{P}^1$, and the holomorphic map $f:X\rightarrow \mathbb{C}^2$, induced by the projection $\mathbb{C}^2\times\mathbb{P}^1\rightarrow \mathbb{C}^2$, satisfies
\begin{itemize}
\item
$f(E)=\{0\}$,
\item
$f:X\setminus E\rightarrow \mathbb{C}^2\setminus\{0\}$ is biholomorphic.
\end{itemize}
\end{rem}
\begin{rem}
If $E$ is assumed to be smooth and the normal bundle of $E$ is negative,
then \cite[Theorem 5.5]{MR79} implies that $X$ is the blow-up of $\mathbb{C}^n$ along some ideal sheaf $\mathcal{I}$ that is supported at $0$.
\end{rem}

From Theorem \ref{thm: modify a regular point text}, we obtain the following result
\begin{thm}[=Theorem \ref{thm:l.b. blow down regular}]\label{thm:l.b. blow down regular text}
Let $\pi: F\ra M$ be a negative holomorphic vector bundle over a compact complex manifold $M$.
 Then the Grauert blow down $(Z, z_0)$ of $F$ is regular at the point $z_0$ if and only if $F$ is a line bundle,
 $M\cong \mathbb P^n$ for some $n$,
 and the line bundles $\pi: F\ra M$ is isomorphic to the tautological line bundle $\mo_{\mathbb P^n}(-1)\ra \mathbb P^n$.
\end{thm}
\begin{proof}
We first prove that the vector bundle $F$ has rank one. Let $f:(F,M)\ra (Z,z_0)$ be the blow-down map.
Noticing that $F$ and $Z$ are complex manifolds of the same dimension, we can infer that $J_{f}^{-1}(0)$,
the zero set of the Jacobian of $f$, is of codimension one. Since $f$ is non-degenerate on $F\backslash M$,
it follows that $J_f^{-1}(0)\subseteq M$. It is then clear that $M$ has codimension one in $F$, i.e. $F$ has rank one.

Now by Theorem \ref{thm: modiy a regular point}, $f:(F,M)\ra (Z,z_0)$ is isomorphic to the blow-up $(\tilde Z, E)\ra (Z, z_0)$ of $Z$ with center $\{z_0\}$,
where $E$ is the exceptional divisor that is isomorphic to $\mathbb P^n$.
We then get an isomorphism as line bundles  between the normal bundle $N_{M/F}\ra M$ of $M$ in $F$ and  the normal bundle $N_{E/\tilde Z}\ra E$ of $E$ in $\tilde Z$.
This implies the conclusion we want to prove since we can naturally identify as vector bundles $N_{M/F}\ra M$ with $F\ra M$,
and identify $N_{E/\tilde Z}\ra E$ with $\mo_{\mathbb P^n}(-1)\ra \mathbb P^n$.
\end{proof}

\section{Uniqueness of irreducible desingularizations of the Grauert blow down of negative vector bundles}\label{sec:uniqueness for negative v.b.}
We prove that the uniqueness theorem is still valid if $(Z,z_0)$ is obtained by blowing down the zero section of a negative line bundle.
\begin{thm}[=Theorem \ref{thm:ample vector bundles}]\label{thm:negative line bundles text}
Let $\pi:L\ra M$ be a negative line over a compact complex manifold $M$, and let $\psi:(L,M)\ra(Z, z_0)$ be the Grauert blow down of $L$.
Then any irreducible desingularization $f:(X,E)\ra (Z, z_0)$ with $E$ being a hypersurface
must be isomorphic to the Grauert blow down $\psi:(L, M)\ra (Z, z_0)$ itself.
\end{thm}
We will use the following lemma
\begin{lem}\label{lem:vanishing order}\cite[Theorem 6.6, Chap II]{Dem}
Let $X$ be a complex manifold, $(A,x)$ be an analytic germ of pure codimension one and let $(A_j,x)$, $j=1,\cdots,M$ be its irreducible components. Then it holds that
\begin{itemize}
\item[(a)]
The ideal of $(A,x)$ is a principal ideal $\mathcal{I}_{A,x}=(h)$ where $h$ is the product of irreducible germs $h_j$ such that $\mathcal{I}_{A_j,x}=(h_j)$.
\item[(b)]
For every $f\in\mathcal{O}_{X,x}$, there is a unique decomposition $f=uh_1^{\alpha_1}\cdot...\cdot h_M^{\alpha_M}$ where $u$ is either invertible or a product of irreducible elements distinct from $h_j$'s, $\alpha_j$ is the vanishing order of $f$ at any point of $A_{j,\mathrm{reg}}\setminus\bigcup_{k\neq j}A_k$.
\end{itemize}
\end{lem}

\begin{rem}\label{rem:local vanishing order}
Let $X$ be a complex manifold, $E$ be an irreducible hypersurface in $X$. For some $x\in E$, $\mathcal{I}_{E,x}=(h)$ and $f\in\mathcal{O}_{X,x}$, we define the order $\mathrm{ord}_{E,x}(f)$ of $f$ along $E$ at $x$ to be the largest integer $\alpha$ such that in the local ring $\mathcal{O}_{X,x}$,
\begin{align*}
f=h^{\alpha}\cdot v
\end{align*}
with $v\in \mathcal{O}_{X,x}$.
It is easy to see that for $g\in\mathcal{O}(X)$, $\mathrm{ord}_{E,x}(g)$ is independent of $x$. Indeed, let $g=uh_1^{\alpha_1}\cdot...\cdot h_M^{\alpha_M}$ be as in Lemma \ref{lem:vanishing order}, since $E_{\mathrm{reg}}$ is connected, we can infer that $\alpha_1=\cdots=\alpha_M=\alpha$ is a constant which is independent of the choice of $x$. The integer $\alpha$ is called the vanishing order $\mathrm{ord}_{E}(g)$ of $g$ along $E$.
\end{rem}

We now give the proof of Theorem \ref{thm:negative line bundles text}.
\begin{proof}
We split the proof into several steps.

\emph{Step 1.}
Taking $k$ large enough such that $(L^*)^k$ is very ample,
by embedding $M$ into the projective space $\mathbb P(H^0(X, L^{-k})^*)$ in the canonical way,
we can identify $M$ with a submanifold of some projective space $\mathbb{P}^{N-1}$ and identify $L^k$ with $\mathcal{O}_{\mathbb{P}^{N-1}}(-1)|_{M}$.

\emph{Step 2.}
As one of the key insights in the argument, we consider the map
\begin{align*}
\Phi_k:&L\rightarrow L^{k}\\
&v\mapsto v^{\otimes k},
\end{align*}
which is a proper $k$-sheeted branched covering with ramification locus $M$.

\emph{Step 3.}
Let $\psi_k:\mathcal{O}_{\mathbb{P}^{N-1}}(-1)\rightarrow\mathbb{C}^N$ be the canonical projection and $Z_k:=\psi_k(L^k)$.
Then $Z_k\subset\mc^N$ is a cone with the origin 0 as the only singularity.
We set
\begin{align*}
g:=\psi_k\circ\Phi_k \circ \psi^{-1}\circ f:X\setminus E\rightarrow Z_k.
\end{align*}
By the Riemann removable singularity theorem, $g$ holomorphically extends to $X$.
We then get the following commutative diagram:
\begin{displaymath}
\xymatrix{
L \ar[rrrr]^{\Phi_k}\ar[d]_{\psi} & & & & L^k \ar[d]^{\psi_k} \\
Z & & & & Z_k\\
& &X\ar[ull]^f \ar[urr]_g && }
\end{displaymath}
We denote by $(z_1,\cdots, z_N)$ the coordinates on $\mc^N$, and set
 $g_j:=g^*z_j$, $j=1,\cdots,N$, as the pull-back of the coordinate functions.
Let
\begin{align*}
m:=\min\{\mathrm{ord}_E(g_j):j=1,\cdots,N\}
\end{align*}
be the minimum of the vanishing orders of $g_1,\cdots, g_N$ along $E$,
which is well defined according to Remark \ref{rem:local vanishing order}.

\emph{Step 4.}
We consider the natural map
\begin{align*}
G:=\psi_k^{-1}\circ g :X\setminus E\rightarrow L^k,\ x\mapsto\Big(g_1(x),\cdots,g_N(x),[g_1(x):\cdots:g_N(x)]\Big).
\end{align*}
Our final aim is to show that the map $\psi^{-1}\circ f:X\backslash E\ra L\backslash M$ can be extended to a biholomorphic map from $X$ to $L$.
We first show that this map can be extended to a holomorphic map from $X$ to $L$.
This will be done by showing that $g$ can be extended to a holomorphic map from $X$ to $L^k$.
Given $x\in E$, assume $\mathcal{I}_{E,x}=(h)$ and $g_j=h^mk_j$ with $k_j\in\mathcal{O}_{X,x}$,
then $G$ extends holomorphically across $x$ if $k_j(x)\neq 0$ for some $1\leq j\leq N$.
We will show that this is always possible for any $x\in E$.
The key is to show that $m=k$.

\emph{Step 5.}
We show in this step that $m\geq k$.
Fix any $x\in E$, assume $\mathcal{I}_{E,x}=(h)$ and $g_j=h^mk_j$ with $k_j\in\mathcal{O}_{X,x}$, $1\leq j\leq N$.
By the definition of $m$, there exists some $k_{j_0}$ that does not vanish identically on $(E,x)$.
Thus we can find $x'\in E$ near $x$ such that $k_{j_0}(x')\neq 0$.
It follows that $G$ can extend holomorphically across $x'$ as given by
$$y\mapsto \left(g_1(y),\cdots, g_N(y), [k_1(y):\cdots:k_N(y)]\right)$$
for $y$ near $x'$.
By this way, we see that there exists a proper analytic subset $Z$ of $E$ such that $G$ holomorphically extends to $X':=X\setminus Z$.
By continuity and Riemann's removable singularity theorem, it follows that
\begin{align*}
F:=\psi^{-1}\circ f: X\setminus E\rightarrow L
\end{align*}
also holomorphically extends to $X'$,
and we still have $\Phi_k\circ F=G$ on $X'$.
Now we get the following communicative diagram.
\begin{displaymath}
\xymatrix{
L \ar[rrrr]^{\Phi_k}\ar[d]_{\psi} & & & & L^k \ar[d]^{\psi_k} \\
Z & & & & Z_k\\
& &X\ar[ull]^f \ar[urr]_g \ar@{.>}[uull]_F \ar@{.>}[uurr]^G && }
\end{displaymath}
From this we can see $m\geq k$ as follows.
Under a trivialization $U\times\mc_u$ and $U\times \mc_v$ of $L$ and $L^k$ respectively,
$\Phi_k$ can be represented as
$$U\times \mc\ra U\times \mc,\ (z,u)\ra (z,v)=(z, u^k).$$
From this it follows that the order of any $h_j:=\Phi_k\circ\psi_k\circ z_j$ along $M$ must be lager than or equal to $k$.
Note that $g_j=F\circ h_j$ on $X'$, so the order of $g_j$ on $X'\cap E$ is also lager than or equal to $k$,
which implies that $m\geq k$.

\emph{Step 6.}
We are going to prove that $m=k$ and $F:X'\rightarrow L$ is surjective.
For showing the surjection of $f$, it suffices to show that
$$F|_{X'\cap E}=G|_{X'\cap E}:X'\cap E\rightarrow M$$
is surjective.
Taking arbitrary $p\in M$, we set $Z_{k,p}:=\psi_k(L^k_p)$, $C_p:=\overline{f^{-1}(\psi(L_p))\setminus E}$,
where $L_p$ and $L^k_p$ denote the fibers of $L$ and $L^k$ at $p$ respectively.
By assumption,  $Z_{k,p}\subset \mc^N$ is a complex line through the origin.
Similar argument as in the proof of Theorem \ref{thm: modify a regular point text} shows that
\begin{itemize}
\item
$C_p\cap E=\{x\}$ for some $x\in E$ is a single point,
\item
$C_p$ is an analytic curve possibly with $x$ as an isolated singularity,
\item
$C_p$ is irreducible at $x$ and $g:C_p\rightarrow Z_{k,p}$ is a $k$-sheeted branched covering.
\end{itemize}

We are going to prove that $G$ extends holomorphically across $x$, and $G(x)=p$.
Let $\sigma:(\mathbb{C},0)\rightarrow (C_p,x)$ be a local parametrization of the germ $(C_p,x)$.
We observe that $g\circ \sigma:(\mathbb{C},0)\rightarrow (Z_{k,p},0)\cong(\mathbb{C},0)$ is a proper $k$-sheeted covering with ramification locus $0$.
So the vanishing order of $g\circ \sigma$ at 0 is $k$.
On the other hand, taking a power-series expansion, one can see that the vanishing order of  $g\circ \sigma$ at 0 is at least $m$.
It follows that $m\geq k$.
Combing with the estimate in the above step, we see $m=k$.
Moreover, if $\mathcal{I}_{E,x}=(h)$ and $g_j=h^kk_j$ with $k_j\in \mo_{X,x}$,
then $k_j(x)\neq 0$ for some $1\leq j\leq N$ (otherwise the vanishing order of $g\circ \sigma$ at 0 would be greater than $k$).
As explained in Step 5, this implies that $G$ extends holomorphically across $x$, and $G(x)=p$.
It follows $G:X'\ra L$ is surjective.

\emph{Step 7.}
The rest of the proof is identical with the last step in the proof of Theorem \ref{thm: modify a regular point text}, we sketch it here for the completeness.
The assumption that $E$ is irreducible implies that $F:X'\rightarrow L$ is non-degenerate everywhere.
The fact that $F$ is a biholomorphism outside $E$ then yields that $F:X'\rightarrow L$ is injective,
hence a biholomorphism. Finally, the compactness of $M$ leads to $X'\cap E=E$ and the proof is complete.
\end{proof}

We can generalize the above theorem to holomorphic vector bundles in certain sense.
\begin{thm}[=Theorem \ref{thm:ample vector bundles}]\label{thm:ample vector bundles text}
Let $f:(X,S)\ra (Z, z_0)$ be an irreducible desingularization of the Grauert blow down $(Z, z_0)$ of a negative line bundle $\pi:L\ra M$
over a compact complex manifold $M$, with $S$ being a smooth submanifold of $X$.
Then there exists a negative vector bundle $F\ra S$ such that $(Z, z_0)$ is isomorphic to the Grauert blow down of $F$
and  $f:(X,S)\ra (Z, z_0)$ is isomorphic to the Grauert blow down $(F, S)\ra (Z, z_0)$.
\end{thm}

\begin{proof}
Let $\varphi:(\tilde{X},\tilde{S})\rightarrow(X,S)$ be the blow-up of $X$ along $S$,
then $f\circ \varphi:(\tilde{X},\tilde{S})\rightarrow (Z,z_0)$ is an irreducible desingularization of $Z$.
Theorem \ref{thm:negative line bundles text} implies that the following diagram communicates
\begin{align*}
\xymatrix{
(\tilde{X},\tilde{S}) \ar[r]^{\varphi}\ar[d]_{\cong}  &(X,S) \ar[d]_{f}\\
(L,M) \ar[r]^{\psi} & (Z,z_0)
}
\end{align*}
where $\psi:(L,M)\ra (Z, z_0)$ is the Grauert blow down map.
We will complete the proof by showing that $(X,S)\cong(N_{S/X},S)$. Indeed, the blow-up of $N_{S/X}$ along $S$ is isomorphic to $N_{\tilde{S}/\tilde{X}}$. Since $\tilde{X}$ is a line bundle over $\tilde{S}$, we can infer that $(\tilde{X},\tilde{S})\cong(N_{\tilde{S}/\tilde{X}},\tilde{S})$. From $X-S\cong\tilde{X}-\tilde{S}$ and $N_{S/X}-S \cong N_{\tilde{S}/\tilde{X}}-\tilde{S}$, we obtain the isomorphism $X-S\cong N_{S/X}- S$. It is easy to see that this isomorphism extends to $S$ identically and $(X,S)\cong (N_{S/X},S)$ follows as desired.
\end{proof}

Theorem \ref{thm:negative line bundles text} has following interesting corollary.

\begin{thm}[=Theorem \ref{thm:blow down determine l.b.}]\label{thm:blow down determine l.b. text}
Let $\pi:L\ra M$ and $\pi':L'\ra M'$ be two negative line bundles over compact complex manifolds.
Let $(Z, z_0)$ and $(Z', z'_0)$ be the Grauert blow downs of $L$ and $L'$ respectively.
If there exist neighborhoods $U$ of $z_0$ in $Z$ and $U'$ of $z'_0$ in $Z'$ such that $U\cong U'$ as complex spaces,
then $M\cong M'$, and $L\cong L'$ as holomorphic line bundles.
\end{thm}
\begin{proof}
Theorem \ref{thm:negative line bundles text} implies that
there is a biholomorphic map $f:(\Omega, M)\ra (\Omega', M')$ for some
neighborhoods $\Omega$ of $M$ in $L$ and $\Omega'$ of $M'$ in $L'$.
The map $f$ imduces as isomorphism of line bundles between the normal bundles $(N_{M/\Omega})$ and $N_{M'/\Omega'}$.
Then the disired conclusion follows since we can identify  $(N_{M/\Omega})$ with $L$, and indify $N_{M'/\Omega'}$ with $L'$.
\end{proof}

\begin{cor}[=Corollary \ref{cor:blow down nbd and normal bundle equi}]
Let $X$ be a complex manifold and $S\subset X$ be a compact complex submanifold
such that the normal bundle $N_{S/X}$ is negative.
Let $(Z,z_0)$ and $(Z',z_0')$ be the normal Grauert blow downs of $(X,S)$ and $(N_{S/X},S)$.
Then $z_0$ and $z_0'$ have isomorphic neighborhoods in $Z$ and $Z'$ if and only if
$S$ has isomorphic neighborhoods in $X$ and $N_{S/X}$.
\end{cor}
\begin{proof}
Let $f:(X,S)\ra (Z,z_0)$ and $f':(N_{S/X},S)\ra (Z',z_0')$ be the Grauert blow downs.
If $\varphi:V\ra V'$ is an isomorphism of neighborhoods of $S$ in $X$ and $N_{S/X}$, we first observe that $\varphi(S)=S$.
Otherwise $f'\circ\varphi:V\ra Z'$ maps $S$ to a compact analytic subset of $Z'$,
which is not a set of isolated points, a contradiction. So we obtain a biholomorphism
$f'\circ\varphi \circ f^{-1}:f(V)\setminus\{z_0\}\ra f'(V')\setminus\{z_0'\}$.
The normality of $Z$ and $Z'$ ensures that $f'\circ\varphi \circ f^{-1}$ extends to an isomorphism of $f(V)$ and $f'(V')$.

On the other hand, if $\psi:U\ra U'$ is an isomorphism of neighborhoods of $z_0$ and $z_0'$ in $Z$ and $Z'$,
then $\psi\circ f:(f^{-1}(U),S)\ra (U',z_0')$ is an irreducible desingularization of $(U',z_0')$.
Theorem \ref{thm:ample vector bundles text} (also its proof) yields that $(f^{-1}(U),S)$ is isomorphic to $(V',S)$,
where $V'$ is some neighborhood of $S$ in $N_{S/X}$, this completes the proof.
\end{proof}

\section{Uniqueness of modifying submanifolds}\label{sec:modify submanifolds}
We consider the uniqueness of modifying a submanifold of a complex manifold.

\begin{thm}\label{thm:mod subumanifolds}
Let $X$ be a complex manifold and $S\subset X$ be a complex submanifold.
Then any irreducible desingularization $f:(Y, E)\ra (X,S)$ of $(X,S)$ must be isomorphic to the
blow up of $X$ along $S$, provided that $E$ is a hypersurface.
\end{thm}
Let us first recall that, the blow-up of $X$ along a closed submanifold $S$ of codimension $s$ is a complex manifold $\tilde{X}$ together with a holomorphic map $\pi:\tilde{X}\rightarrow X$ such that
\begin{itemize}
\item
$\tilde{S}:=\pi^{-1}(S)$ is a smooth hypersurface in $\tilde{X}$ and the restriction $\pi:\tilde{S}\rightarrow S$ is isomorphic to the projectivized normal bundle $\textbf{P}(N_{S/X})\rightarrow S$,
\item
$\pi:\tilde{X}\setminus\tilde{S}\rightarrow X\setminus S$ is a biholomorphism.
\end{itemize}
For each $x_0\in S$, we take a neighbourhood $V$ with a coordinate chart $\tau(z)=(z_1,\cdots,z_n):V\rightarrow \mathbb{C}^n$ centered at $x_0$, such that $\tau(V)=\mathbb{B}^s\times \mathbb{B}^{n-s}$ for some balls $\mathbb{B}^s \subseteq \mathbb{C}^s$, $\mathbb{B}^{n-s}\subseteq\mathbb{C}^{n-s}$ and $S\cap V=\{z_1=\cdots=z_s=0\}$. Then $\frac{\partial}{\partial z_1}|_{S\cap V},\cdots,\frac{\partial}{\partial z_s}|_{S\cap V}$ form a holomorphic frame of $N_{S/X}|_{S\cap V}$, and we denote by $(\xi_1,\cdots,\xi_s)$ the corresponding coordinate along fibers of $N_{S/X}$. Thus $(\xi_1,\cdots,\xi_s,z_{s+1},\cdots,z_n)$ is a coordinate chart on $N_{S/X}|_{S\cap V}$. For each $j=1,\cdots,s$, we set
\begin{align*}
\tilde{U}_j:=\{z\in V\setminus S:z_j\neq0\}\cup\{(z,[\xi])\in\textbf{P}(N_{S/X})|_{S\cap V}:\xi_j\neq0\}
\end{align*}
Then $\tilde{U}_j$ forms an open covering of $\pi^{-1}(U)$. Moreover, the holomorphic coordinate $\tilde{\tau}_j:\tilde{U}_j\rightarrow \mathbb{C}^n$ is given by
\begin{align*}
\tilde{\tau}_j(z)&=\Big(\frac{z_1}{z_j},\cdots,\frac{z_{j-1}}{z_j},z_j,\frac{z_{j+1}}{z_j},\cdots,\frac{z_s}{z_j},z_{s+1},\cdots,z_n \Big),\ z\in V\setminus S,\\
\tilde{\tau}_j(z,[\xi])&=\Big(\frac{\xi_1}{\xi_j},\cdots,\frac{\xi_{j-1}}{\xi_j},0,\frac{\xi_{j+1}}{\xi_j},\cdots,\frac{\xi_s}{\xi_j},z_{s+1},\cdots,z_n\Big),\ (z,[\xi])\in \textbf{P}(N_{S/X})|_{S\cap V}.
\end{align*}

\begin{proof}
The basic idea is similar to the argument of Theorem \ref{thm: modify a regular point text}.
For each $x_0\in S$, we can take a local coordinate neighbourhood $(V,z_1,\cdots,z_n)$ as above
and denote $f_j:=f^*z_j$, $j=1,\cdots,n$. Those are holomorphic functions on $U:=f^{-1}(V)$ such that $E\cap U=\{f_1=\cdots=f_s=0\}$.
Let $L$ be an analytic curve passing $x_0$ and intersecting $S$ transversally.
In local coordinate, we can choose $L$ to be the image of
\begin{align*}
\gamma:\mathbb{B}^1 &\rightarrow \mathbb{B}^{s}\times\mathbb{B}^{n-s},\\
 t \ &\mapsto \gamma(t)=(\gamma_1(t),\cdots,\gamma_s(t),0,\cdots,0),
\end{align*}
where $\gamma_1(0)=\cdots=\gamma_s(0)=0$, and $\frac{\partial\gamma_j}{\partial t}(0)\neq0$ for some $1\leq j\leq s$.
Let $C_0:=f^{-1}(L)$ and $C_1:=\overline{C_0\setminus E}$.
Using the same argument in Theorem \ref{thm: modify a regular point text}, we obtain that for some $p_L\in E$,
\begin{itemize}
\item
$C_1\cap E=\{p_L\}$,
\item
$(C_1,p_L)$ is regular,
\item
$f|_{C_1}$ is non-degenerate at $p_L$, i.e. $df|_{C_1}(p_L)\neq0$.
\end{itemize}
We may assume that $df_1(p_L)\neq0$, so $(f_1^{-1}(0),p_L)$ is a smooth hypersurface germ.
Since $(E,p_L)=(f_1^{-1}(0)\cap\cdots\cap f_s^{-1}(0),p_L)$,
we infer that $(E,p_L)=(f_1^{-1}(0),p_L)$ and in particular $E$ is regular at $p_L$,
otherwise for some $j>1$, $f_j^{-1}(0)$ does not vanish on $(f_1^{-1}(0),p_L)$. This leads to $\dim(E,p_L)\leq(n-2)$, a contradiction.
It follows that $\frac{f_2}{f_1},\cdots,\frac{f_s}{f_1}$ can be viewed as holomorphic functions near $p_L$.
If we consider the map
\begin{align*}
\tilde{\tau}_1\circ \pi^{-1}   \circ f:U\setminus E &\rightarrow \tilde{\tau}_1(\tilde{U}_1)\\
y\ \ \ &\mapsto \left(f_1(y),\frac{f_2(y)}{f_1(y)},\cdots,\frac{f_s(y)}{f_1(y)},f_{s+1}(y),\cdots,f_n(y)\right).
\end{align*}
Then it holomorphically extends to $p_L$ and $p_L$ is mapped to $\left(0,\frac{\gamma_2'(0)}{\gamma_1'(0)},\cdots,\frac{\gamma_s'(0)}{\gamma_1'(0)},0,\cdots,0\right)$,
which corresponds to $(x_0,[L])\in\textbf{P}(N_{S/X})_{x_0}$. This implies that
\begin{align*}
F:=\pi^{-1}\circ f:Y\setminus E\rightarrow \tilde{X}
\end{align*}
holomorphically extends to $p_L$ and $F(p_L)=(x_0,[L])$.
Now for each $q_0\in E$, we can take $(V,z)$ being a local chart centered at $f(q_0)$ such that $S\cap V=\{z_1=\cdots=z_s=0\}$.
If some of $df_1(q_0),\cdots,df_s(q_0)$ does not vanish, we see that $F$ holomorphically extend to $q_0$ in the same way.
It is easy to see that the analytic subset
\begin{align*}
\{q\in E:df_1(q)=\cdots=df_s(q)=0\}
\end{align*}
of $f^{-1}(V)$ coincides with
\begin{align*}
\{q\in E:(df)_q(T_qY)\subseteq T_{f(q)}S\}
\end{align*}
Therefore we obtain a holomorphic map $F:Y'\rightarrow \tilde{X}$ such that
\begin{itemize}
\item
$F$ is surjective,
\item
$ Y'\setminus E=Y\setminus E$,
\item
$Y'\cap E=E\setminus\{q\in E:(df)_q(T_qY)\subseteq T_{f(q)}S\}=(E)_{\mathrm{reg}}\setminus\{q\in E:(df)_q(T_qY)\subseteq T_{f(q)}S\}$.
\end{itemize}

The rest of the proof is similar as that of Theorem \ref{thm: modify a regular point text}.
First, we prove that $F$ is non-degenerate on $Y'$. The analytic subset of $Y'$
\begin{align*}
J_F^{-1}(0):=\{y\in Y': \mathrm{the}\ \mathrm{Jacobian}\ \mathrm{of}\ F\ \mathrm{vanishes}\ \mathrm{at}\ y\}
\end{align*}
is either empty or has pure dimension $(n-1)$. The assumption $E$ is irreducible yields that $Y'\cap E$ is also irreducible.
Since $J^{-1}_F(0)\subseteq Y'\cap E$, we can infer $J_F^{-1}(0)=Y'\cap E$ if $J_F^{-1}(0)$ is non-empty.
On the other hand, $F$ maps $Y'\cap E$ onto $\tilde{S}$.
Noticing that $X'\cap X_0$ and $\tilde{S}$ are complex manifolds of the same dimension,
we can infer $F|_{Y'\cap E}$ is non-degenerate at some point on $Y'\cap E$, saying $p_0$.
In particular, $F|_{Y'\cap E}$ is injective on $Y'\cap E$ near $p_0$, which makes $F$ is injective in $Y'$ near $p_0$.
\cite[Theorem 8.5]{FG02} implies that $J_F(p_0)\neq0$, a contradiction. Therefore, we demonstrate that $J_F^{-1}(0)$ is empty, i.e. $F$ is non-degenerate on $Y'$.

We then prove that $F:Y'\rightarrow \tilde{X}$ is injective.
Suppose to the contrary, for some $p_1\neq p_2\in Y'\cap E$, $F(p_1)= F(p_2)=q\in \tilde{S}$.
Then there exists some neighbourhood $U_i$ of $p_i$, $V$ of $q$, such that $F:U_i\rightarrow V$ is biholomorphism.
This contradicts with the fact that $F$ is injective on $Y\setminus E$.
Now $F:Y'\rightarrow \tilde{X}$ is a biholomorphism and we can denote its inverse by $G:\tilde{X}\rightarrow Y'$.

Finally, we prove that that $Y'\cap E=E$. Taking $p_0\in E$, there exists $p_k\in Y'\cap E$ converges to $p_0$.
From the construction of $F$, we see that $F(p_k)\in\textbf{P}(N_{S/X})_{f(p_k)}$ approaching the the fiber $\textbf{P}(N_{S/X})_{f(p_0)}$.
Up to a subsequence, we may assume $\lim F(p_k)=\tilde{x_0}\in \textbf{P}(N_{S/X})_{f(p_0)}$. Since $G$ is continuous,
we see that $G(\tilde{x}_0)=\lim G( F(p_k))=\lim p_k=p_0$.
This implies that $E=G(\tilde{S})=Y'\cap E$ as desired.
We draw to the conclusion that $f:(Y,E)\ra(X,S)$ is isomorphic to $\pi:(\tilde{X},\tilde{S})\ra(X,S)$.
\end{proof}

	\end{document}